\documentclass[reqno, 12pt]{amsart}
\usepackage{amssymb}
\usepackage{amsmath}
\usepackage[mathscr]{euscript}
\usepackage{mathtools}
\usepackage[top=2cm,bottom=3cm,left=2.5cm,right=2.5cm]{geometry}
\usepackage{hyperref}

\theoremstyle{plain}
\newtheorem{theorem}{Theorem}
\newtheorem{corollary}{Corollary}
\newtheorem{lemma}{Lemma}

\theoremstyle{definition}
\newtheorem{definition}{Definition}
\newtheorem{example}{Example}
\newtheorem{remark}{Remark}

\numberwithin{equation}{section}

\begin{document}

\title[Conditional Lipschitz shadowing]{Conditional Lipschitz shadowing \\for ordinary
differential equations}

\title[Conditional shadowing]{Conditional Lipschitz shadowing \\for ordinary
differential equations}

\author[L. Backes]{Lucas Backes}
\address{\noindent Departamento de Matem\'atica, Universidade Federal do Rio Grande do Sul, Av. Bento Gon\c{c}alves 9500, CEP 91509-900, Porto Alegre, RS, Brazil.}
\email{lucas.backes@ufrgs.br} 
\author[D. Dragi\v cevi\'c]{Davor Dragi\v cevi\'c}
\address{Faculty of Mathematics, University of Rijeka, Radmile Matej\v ci\' c 2, 51000 Rijeka, Croatia}
\email{ddragicevic@math.uniri.hr}
\author[M. Onitsuka]{Masakazu Onitsuka}
\address{Department of Applied Mathematics, Faculty of Science, Okayama University of Science, Ridai-cho 1-1, Okayama 700-0005, Japan}
\email{onitsuka@ous.ac.jp}
\author[M. Pituk]{Mih\'aly Pituk}
\address{Department of Mathematics, University of Pannonia, Egyetem \'ut 10, 8200 Veszpr{\'e}m, Hungary;
ELKH-ELTE Numerical Analysis and Large Networks Research Group}
\email{pituk.mihaly@mik.uni-pannon.hu}

%\classification[37D99]{34D10}


\maketitle

\vskip10pt
\begin{abstract}
We introduce the notion of conditional Lipschitz shadowing, which 
does not aim to shadow every pseudo-orbit, but only those which belong to a certain prescribed set. We establish two types of sufficient conditions under which certain non\-auto\-nomous ordinary differential equations have such a property. The first criterion applies to a semilinear differential equation provided that its linear part is hyperbolic and the nonlinearity is small in a neighborhood of the prescribed set. The second criterion requires that the logarithmic norm of the derivative of the right-hand side with respect to the state variable is uniformly negative in a neighborhood of the prescribed set. 
The results are applicable to important classes of model equations including the logistic equation, whose conditional shadowing has recently been studied. Several examples are constructed showing that the obtained conditions are optimal.
\end{abstract}

\maketitle

\section{Introduction}
\label{sec1}

One of the main properties of chaotic dynamical systems is that different orbits starting close by will move apart as time evolves and  
this divergence can be very fast. In particular, whenever performing numerical experiments, small numerical errors, which are inherent to the process due to computer round off errors, tend to grow. Thus, a central question 
when performing such numerical experiments is if the resulting behaviour presented by the computer reflects the dynamics of the actual system. It turns out that for some classes of chaotic systems, like the hyperbolic ones, even though a numerical orbit containing round off errors will diverge rapidly from the true orbit with the same initial condition, there exists a different true orbit 
which stays near the noisy orbit.
Systems exhibiting this property are said to have the \emph{shadowing property}. In other words, a dynamical system  
has the shadowing property if close to its approximate orbits we can find exact ones. The main objective of the present paper is to present sufficient conditions under which a general class of nonautonomous nonlinear ordinary differential equations exhibits a new variant of 
shadowing property, the so-called \emph{conditional Lipschitz shadowing property}, 
defined below (see Definition~\ref{def: conditional shadowing}).

Let $\mathbb R^n$ and $\mathbb R^{n\times n}$ denote the $n$-dimensional space of real column vectors and the space of $n\times n$ matrices with real entries, respectively. The symbol $|\cdot|$ denotes any convenient norm on $\mathbb R^n$ and the associated induced norm on $\mathbb R^{n\times n}$. Consider the ordinary differential equation 
\begin{equation}\label{ODE}
x'=g(t,x),
\end{equation}
where $g\colon [0, \infty) \times \mathbb R^n \to \mathbb R^n$ is a continuous. We are interested in the noncontinuable solutions of~\eqref{ODE} starting at $t=0$. It is known that these solutions are defined on intervals of type $[0,\tau)$, where $\tau\in(0,\infty]$ may depend on the solution~$x$. For this reason, we will consider pseudosolutions (approximate solutions) of~\eqref{ODE} on intervals of the same type in the following sense.
%\begin{definition}
Given $\tau\in(0,\infty]$, by a \emph{pseudosolution of Eq.~\eqref{ODE} on $[0,\tau)$}, we mean any continuously differentiable function $y\colon[0,\tau)\rightarrow\mathbb R^n$ such that 
\begin{equation}
\label{pseudodef}
	\sigma_y:=\sup_{0\leq t<\tau}|y'(t)-g(t,y(t))|<\infty.
\end{equation}
The function $e_y\colon[0,\tau)\rightarrow[0,\infty)$ defined by
\begin{equation}
\label{error}
e_y(t):=|y'(t)-g(t, y(t))|\qquad\text{for $t\in[0,\tau)$,}
\end{equation}
will be called the \emph{error function} and the quantity $\sigma_y$ is the \emph{maximum error} corresponding to~$y$.
%\end{definition}
Let us recall the definition of the standard Lipschitz shadowing property.

\begin{definition}

We say that Eq.~\eqref{ODE} has the \emph{Lipschitz shadowing property} if there exist $\varepsilon_0>0$ and $\kappa>0$ with the following property: 
if $0<\varepsilon\leq\varepsilon_0$ and $y$ is a pseudosolution of~\eqref{ODE} on $[0,\tau)$ for some $\tau\in(0,\infty]$ such that 
$\sigma_y\leq\varepsilon$, then Eq.~\eqref{ODE} has a solution $x$ on $[0,\tau)$ satisfying
\begin{equation}
\label{closesol}
	\sup_{0\leq t<\tau}|x(t)-y(t)|\leq \kappa\varepsilon.
\end{equation}
\end{definition}
%\begin{remark}
The concept of Lipschitz shadowing is closely related to the stronger notion of \emph{Hyers--Ulam stability}
(\emph{Ulam stability}), 
which requires that the condition in the above definition is satisfied for every $\varepsilon>0$. For a related concept in the theory of smooth dynamical systems, see~\cite[Definition 1.5]{Pil}.
%\end{remark}

Consider the semilinear differential equation
\begin{equation}
\label{pode}
x'=A(t)x+f(t,x)
\end{equation}
as a perturbation of the linear equation
\begin{equation}
\label{lde}
x'=A(t)x,
\end{equation}
where $A\colon[0,\infty)\rightarrow\mathbb R^{n\times n}$  and $f\colon[0,\infty)\times \mathbb R^{n}\rightarrow\mathbb R^n$ are continuous. 
The following known result provides a sufficient condition for the Lipschitz shadowing property of~\eqref{pode} (see~\cite[Theorem 6]{BD}).

\begin{theorem}
\label{bdthm}
Suppose that the linear equation~\eqref{lde} has an exponential dichotomy on $[0,\infty)$ (see Definition~\ref{ED}) and there exists $L\geq0$ such that
\begin{equation}
	\label{globlip}
|f(t,x_1)-f(t,x_2)|\leq L|x_1-x_2|\qquad\text{for all $t\geq0$ and $x_1, x_2\in\mathbb R^n$.}
\end{equation}
If~$L$ is sufficiently small, then Eq.~\eqref{pode} has the Lipschitz shadowing property.
\end{theorem}
%For the definition of exponential dichotomy, see Sec.~\ref{sec2}. 
Note that Theorem~\ref{bdthm} can be extended to the more general class of delay differential equations (see~\cite[Theorem~2.3]{BDPS}). For further related results about shadowing and Hyers--Ulam stability of ordinary differential equations and their discrete counterparts,  
see~\cite{BD0, BD, BD2, BDPS, BBT, BCDP, BLO, BOST} and references therein.

Some recent  studies (see \cite{Onit, PRV}) have been concerned with the shadowing (Hyers--Ulam stability) of the scalar \emph{logistic equation}
\begin{equation}
\label{logisteq}
x'=x(ax+b),\qquad a,b\in\mathbb R\setminus\{0\},
\end{equation}
which is a particular case of Eq.~\eqref{pode} when $n=1$, $A(t)=b$ and  
$f(t,x)=ax^2$. Note that Theorem~\ref{bdthm} cannot be applied to Eq.~\eqref{logisteq} because~$f$ does not satisfy the global Lipschitz condition~\eqref{globlip}. In~\cite{PRV}, it has been shown that  
Eq.~\eqref{logisteq} with $a=-1$ and $b=1$ is not Hyers--Ulam stable, but certain approximate solutions still can be shadowed by true solutions.
 
Motivated by this observation, we introduce the notion of conditional Lipschitz shadowing, 
%property, 
which does not require the validity of the Lipschitz shadowing property for all pseudosolutions, but only for those which belong to a given set $H\subset\mathbb R^n$.
\begin{definition}\label{def: conditional shadowing}
Let $H$ be a nonempty subset of $\mathbb R^n$. 
We say that Eq.~\eqref{ODE} has the \emph{conditional Lipschitz shadowing property in~$H$} if there exist $\varepsilon_0>0$ and $\kappa>0$ with the following property: if $0<\varepsilon\leq\varepsilon_0$ and $y$ is a pseudosolution of~\eqref{ODE} on $[0,\tau)$ for some $\tau\in(0,\infty]$ such that 
$\sigma_y\leq\varepsilon$ and $y(t)\in H$ for all $t\in[0,\tau)$, then Eq.~\eqref{ODE} has a solution $x$ on $[0,\tau)$ satisfying~\eqref{closesol}.
\end{definition}

Evidently, the standard Lipschitz shadowing property is a special case of the conditional  Lipschitz shadowing property with $H=\mathbb R^n$. A different concept of conditional shadowing for discrete nonautonomous systems in a Banach space has recently been introduced by Pilyugin~\cite{Pil2}.

In this paper, we establish two types of sufficient conditions under which certain classes of ordinary differential equations have the conditional Lipschitz shadowing property in a given set $H\subset\mathbb R^n$.

In Sec.~\ref{sec2}, we consider the semilinear differential equation~\eqref{pode}. The main result of this part is formulated in Theorem~\ref{T1}, which is a generalization of Theorem~\ref{bdthm} to the case of conditional Lipschitz shadowing. It says that Eq.~\eqref{pode} has the conditional Lipschitz shadowing property in a prescribed set~$H$ whenever its linear part has an exponential dichotomy and the nonlinearity $f(t,x)$ is Lipschitz in~$x$ in a neighborhood of~$H$ (uniformly in~$t$) with a sufficiently small Lipschitz constant. The smallness condition on the Lipschitz constant can be expressed in terms the dichotomy constants of the linear part. The importance of the obtained sufficient condition will be shown by an application to a scalar logistic equation. In Example~\ref{EXM}, we show that our choice of pseudosolutions is optimal.

 In Sec.~\ref{lnorms}, we present sufficient conditions for the conditional Lipschitz shadowing of Eq.~\eqref{ODE} in terms of the logarithmic norm of $g_x(t,x)$, the partial derivative of $g$ with respect to~$x$. The  \emph{logarithmic norm} (\emph{Lozinski\u{\i} measure}) of 
 a square matrix  $A\in \mathbb R^{n\times n}$ is defined by
	\begin{equation}
	\label{lognorm}
\mu(A):=\lim_{h\to0+}\frac{|I+hA|-1}{h}\qquad \text{for $A\in\mathbb R^{n\times n}$},
\end{equation}
where $I$ is the identity matrix in $\mathbb R^{n\times n}$.
It should be noted that~$\mu$ is not a norm, since it can take negative values. In the scalar case ($n=1$), we have that $\mu(A)=A$. The values of $\mu(A)$ for the standard norms in~$\mathbb R^n$ can be given explicitly (see Sec.~\ref{lnorms}). 
 The main result of this section, Theorem~\ref{globcrit}, says that if $\mu(g_x(t,x))$ is uniformly negative (bounded away from zero) for all $t\geq0$ and $x$ in a neighborhood of the given set $H\subset\mathbb R^n$, then Eq.~\eqref{ODE} has the conditional Lipschitz shadowing property in~$H$. 
 To the best of our knowledge, this criterion has no previous analogue. It gives a new result even in the case of the standard Lipschitz shadowing. The importance and the sharpness of the assumptions 
%of Theorem~\ref{globcrit}
will be shown in a special case of the Kermack--McKendrick equation from epidemiology.

\section{Conditional Lipschitz shadowing via exponential dichotomy}
\label{sec2}

In this section, we give sufficient conditions for the conditional Lipschitz shadowing of the semilinear equation~\eqref{pode}.

Let $\varPhi$ be a fundamental matrix solution of the linear equation~\eqref{lde} so that its transition matrix $T(t,s)$ is given by
\begin{equation*}
\label{transmat}
T(t,s):=\varPhi(t)\varPhi^{-1}(s)\qquad\text{for $t,s\in[0,\infty)$}.
\end{equation*}
\begin{definition}\label{ED}
We say that Eq.~\eqref{lde} has an \emph{exponential dichotomy} on~$[0,\infty)$ if there exist a family of projections $\left(P(t)\right)_{t\geq0}$ in~$\mathbb R^{n\times n}$ and  constants $N, \lambda >0$ such that
\begin{gather} 
P(t)T(t,s)=T(t,s)P(s)\qquad\text{for all $t,s\in[0,\infty)$},\label{ed0}\\ 
|T(t,s)P(s)|\leq Ne^{-\lambda(t-s)}\qquad\text{whenever $t\geq s\geq 0$}\label{ed1}
\end{gather}
and
\begin{equation}\label{ed2}
|T(t,s)(I-P(s))|\leq Ne^{-\lambda(s-t)}\qquad\text{whenever $0\leq t\leq s$}.
\end{equation}

If, in addition, $P(t)=I$ ($P(t)=0$) identically for $t\ge 0$, we say that Eq.~\eqref{lde} admits an \emph{exponential contraction (exponential expansion)} on $[0,\infty)$.
\end{definition}

Throughout the paper, for $x\in\mathbb R^n$ and $\delta>0$, $B_\delta(x)$ will denote the closed $\delta$-neighborhood of~$x$ in $\mathbb R^n$ given by $B_\delta(x):=\{\,x\in\mathbb R^n:|y-x|\leq\delta\,\}$. For $\emptyset\neq H\subset\mathbb R^n$, the closed $\delta$-neighborhood of~$H$ is defined by
\begin{equation*}
\mathcal N_\delta(H):=\bigcup_{x\in H} B_\delta(x).
\end{equation*}

The main result of this section is the following generalization of Theorem~\ref{bdthm} to the case of conditional Lipschitz shadowing.
\begin{theorem}\label{T1}
Let $\emptyset\neq H\subset\mathbb R^n$. Suppose that Eq.~\eqref{lde} has an exponential dichotomy on $[0,\infty)$ and that there exist $\delta, L>0$ such that
\begin{equation}
	\label{lipsch}
|f(t,x_1)-f(t,x_2)|\leq L|x_1-x_2|,\qquad\text{for all $t\geq0$ and $x_1,x_2\in{\mathcal N}_{\delta}(H)$.}
\end{equation}
If
\begin{equation}
\label{smallcond}
L<\frac{\lambda}{2N},
\end{equation}
where $N, \lambda >0$ are as in Definition~\ref{ED},
then Eq.~\eqref{pode} has the conditional Lipschitz shadowing property in~$H$.
\end{theorem}

\begin{proof}
From \eqref{smallcond}, it follows that the equation
\begin{equation}\label{kappa}
\frac{2N}{\lambda}(L\kappa+1)=\kappa
\end{equation}
has a unique solution $\kappa$  given by
\begin{equation*}
\kappa:=\bigg(\frac{\lambda}{2N}-L\biggr)^{-1}=\frac{2N}{\lambda-2NL}>0.
\end{equation*}
Set $\varepsilon_0:=\kappa^{-1}\delta>0$.
Suppose that 
$0<\varepsilon\leq \varepsilon_0$ and $y$ is a pseudosolution of~\eqref{pode} on $[0,\tau)$ for some $\tau\in(0,\infty]$ such that 
$\sigma_y\leq\varepsilon$ and $y(t)\in H$ for all $t\in[0,\tau)$.  Observe that the transformation
$z=x-y$ reduces Eq.~\eqref{pode} to the equation
\begin{equation}
\label{zode}
z'=A(t)z+f(t,y(t)+z)-f(t,y(t))+h_y(t),	
\end{equation}
where
\begin{equation}
\label{beq}
h_y(t):=A(t)y(t)+f (t,y(t))-y'(t)\qquad\text{for $t\in[0,\tau)$}.
\end{equation}
Clearly, $|h_y(t)|=e_y(t)$ for $t\in[0,\tau)$, where $e_y$ is the error function corresponding to the pseudosolution~$y$ of Eq.~\eqref{pode}. Hence
\begin{equation}
\label{hsmall}
	\sup_{0\leq t<\tau}|h_y(t)|=\sigma_y\leq\varepsilon.
\end{equation}
Let $\mathcal C_b:=C([0,\tau),\mathbb R^n)$ denote the Banach space of bounded and continuous functions $z\colon [0, \tau) \to \mathbb R^n$  equipped with the supremum norm,
\begin{equation*}
\|z\|=\sup_{t\in[0,\tau)}|z(t)|,\qquad z\in\mathcal C_b.
 \end{equation*}
 Set
\begin{equation*}
\label{Sdef}
S:=\{\,z\in\mathcal C_b:\|z\|\leq\kappa\varepsilon\,\}.
\end{equation*}
Clearly, $S$ is a nonempty and closed subset of~$\mathcal C_b$. For $z\in S$ and $t\in[0,\tau)$, define
\begin{align*}
(\mathcal F z)(t)&:=\int_0^t T(t,s)P(s)\left[\,f(s,y(s)+z(s))-f(s,y(s))+h_y(s)\,\right]\,ds\\
&\qquad-\int_t^\tau T(t,s)(I-P(s))\left[\,f(s,y(s)+z(s))-f(s,y(s))+h_y(s)\,\right]\,ds.
\end{align*}
Take an arbitrary $z\in S$ and $s\in[0,\tau)$. Then,
\begin{equation*}
|(y(s)+z(s))-y(s)|=|z(s)|\leq\|z\|\leq\kappa\varepsilon\leq\kappa\varepsilon_0=\delta.
\end{equation*}
Therefore,
\begin{equation}
\label{incl}
y(s)\in H\subset\mathcal N_\delta(H)\qquad\text{and}\qquad
y(s)+z(s)\in B_\delta(y(s))\subset\mathcal N_\delta(H),
\end{equation}
which, together with~\eqref{lipsch}, implies that
\begin{equation*}
|f(s,y(s)+z(s))-f(s,y(s))|\leq L|z(s)|\leq L\|z\|\leq L\kappa\varepsilon.
\end{equation*}
This, combined with~\eqref{ed1}, \eqref{ed2} and~\eqref{hsmall} (see also~\eqref{kappa}) yields that 
\begin{equation}\label{00}
\begin{split}
&|(\mathcal F z)(t)|\leq \int_0^t |T(t,s)P(s)|
\left(\,|f(s,y(s)+z(s))-f(s,y(s))|+|h_y(s)|\,\right)\,ds\\
&\qquad+\int_t^\tau |T(t,s)(I-P(s))|\left(\,|f(s,y(s)+z(s))-f(s,y(s))|+|h_y(s)|\,\right)\,ds\\
&\leq\left(L\kappa\varepsilon+\varepsilon\right)
\biggl(\int_0^t Ne^{-\lambda(t-s)}\,ds
+\int_t^\infty Ne^{-\lambda(s-t)}\,ds\biggr)\\
&\leq\frac{2N}{\lambda}\left(L\kappa+1\right)\varepsilon=\kappa\varepsilon,
\end{split}
\end{equation}
 for $z\in S$ and $t\in[0,\tau)$.
We conclude that  $\mathcal F z$ is well-defined and $\mathcal F(S)\subset S$.

Let $z_1, z_2\in S$. In view of~\eqref{incl}, we have that $y(s)+z_j(s)\subset\mathcal N_\delta(H)$ for $s\in[0,\tau)$ and $j=1,2$. Hence,
\begin{equation*}
|f(s,y(s)+z_1(s))-f(s,y(s)+z_2(s))|\leq L|z_1(s)-z_2(s)|\leq L\|z_1-z_2\|,
\end{equation*}
for $s\in [0, \tau)$. 
Consequently, 
\begin{equation}\label{11}
\begin{split}
&|(\mathcal F z_1)(t)-(\mathcal F z_2)(t)| \\
&\leq \int_0^t |T(t,s)P(s)|
\left(\,|f(s,y(s)+z_1(s))-f(s,y(s)+z_2(s))|\,\right)\,ds\\
&\qquad+\int_t^\tau |T(t,s)(I-P(s))|\left(\,|f(s,y(s)+z_1(s))-f(s,y(s)+z_2(s))|\,\right)\,ds\\
&\leq L\|z_1-z_2\|
\biggl(\int_0^t Ne^{-\lambda(t-s)}\,ds
+\int_t^\infty Ne^{-\lambda(s-t)}\,ds\biggr)\\
&\leq\frac{2N}{\lambda}L\|z_1-z_2\|,
\end{split}
\end{equation}
for $t\in [0, \tau)$.
Therefore, for all $z_1,z_2\in S$,
\begin{equation*}
\|\mathcal F z_1-\mathcal F z_2\|\leq q\|z_1-z_2\|\qquad\text{with $q:=\frac{2N}{\lambda}L<1$}.
\end{equation*}
Thus, $\mathcal F\colon S\rightarrow S$ is a contraction and it has a unique fixed point~$z$ in~$S$. It follows by differentiation that  $z$  is a solution of Eq.~\eqref{zode} on $[0,\tau)$. Moreover, $z\in S$ implies that
\begin{equation*}
\sup_{t\in[0,\tau)}|z(t)|=\|z\|\leq\kappa\varepsilon.
\end{equation*}
Therefore, $x=z+y$ is a solution of Eq.~\eqref{pode} on $[0,\tau)$ with the desired property~\eqref{closesol}. The proof of the theorem  is complete.
\end{proof}
\begin{remark}\label{old}
Theorem~\ref{bdthm} is a corollary Theorem~\ref{T1} with $H=\mathbb R^n$. \end{remark}

\begin{remark}
If $D$ is a nonempty set in $\mathbb R^n$, then its \emph{convex hull}, denoted by $\operatorname{conv}(D)$, is the smallest convex set in $\mathbb R^n$ which contains~$D$. A sufficient condition for the Lipschitz condition~\eqref{lipsch} to hold is that $f$ is continuously differentiable and
\begin{equation}
\label{smallderiv}
|f_x(t,x)|\leq L\qquad\text{for all $t\geq0$ and $x\in\operatorname{conv}\left(\mathcal N_\delta(H)\right)$}.\end{equation}
\end{remark}

The following result is an improvement of Theorem~\ref{T1} in the particular case when Eq.~\eqref{lde} admits an exponential contraction or exponential expansion. It shows that in these settings the smallness condition~\eqref{smallcond} for the Lipschitz constant~$L$ can be weakened.
\begin{theorem}
\label{T2}
Let $\emptyset\neq H\subset\mathbb R^n$. Suppose that Eq.~\eqref{lde} has an exponential contraction or exponential expansion on $[0,\infty)$ and there exist $\delta, L>0$ such that~\eqref{lipsch} holds.
If
\begin{equation}
\label{smallcond2}
L<\frac{\lambda}{N},
\end{equation}
then Eq.~\eqref{pode} has the conditional Lipschitz shadowing property in~$H$.
\end{theorem}

\begin{proof}
The proof proceeds in a similar manner as the proof of Theorem~\ref{T1}.  Take $\kappa >0$ such that 
\begin{equation*}
\frac{N}{\lambda}(L\kappa+1)=\kappa.
\end{equation*}
Let $\varepsilon_0>0$, $y$, $S$ and $\mathcal F$ be as in the proof of Theorem~\ref{T1}. By arguing as in~\eqref{00}  (recall that either $P(t)\equiv I$ or $P(t)\equiv0$), we have that 
\begin{equation*}
 | (\mathcal F z)(t)| \le \frac{N}{\lambda}(L\kappa +1)\varepsilon=\kappa \varepsilon
\end{equation*}
for $t\in [0, \tau)$ and $z\in S$. Moreover, by similar estimates as   in~\eqref{11}, we conclude that
\begin{equation*}
|(\mathcal F z_1)(t)-(\mathcal F z_2)(t)|\le \frac{N}{\lambda}L\|z_1-z_2\|,
\end{equation*}
for $t\in [0, \tau)$ and $z_1, z_2 \in S$. Now one can complete the proof by the same arguments as in the proof of Theorem~\ref{T1}.
\end{proof}

The following consequence of Theorems~\ref{T1}
and~\ref{T2} gives sufficient conditions under which Eq.~\eqref{pode} has the conditional Lipschitz shadowing property in a 
given neighborhood of the origin.

\begin{corollary}
\label{thm2}
Let $\rho>0$. Suppose that Eq.~\eqref{lde} has an exponential dichotomy on $[0,\infty)$ and there exist $\delta, L>0$  such that
\begin{equation}
	\label{balllipsch}
|f(t,x_1)-f(t,x_2)|\leq L|x_1-x_2|\qquad\text{for all $t\geq0$ and $x_1, x_2\in B_{\rho+\delta}(0)$}.
\end{equation}
Then, \eqref{smallcond} implies that Eq.~\eqref{pode} has the conditional Lipschitz shadowing property in $B_\rho(0)$. Moreover, if instead of the existence of  an exponential dichotomy, we assume that Eq.~\eqref{lde} has an exponential contraction or exponential expansion on $[0,\infty)$, then 
the conditional Lipschitz shadowing property of Eq.~\eqref{pode} in $B_\rho(0)$ holds under the weaker condition~\eqref{smallcond2}.
\end{corollary}
\begin{proof}Let $H=B_{\rho}(0)$. Then 
$\mathcal N_\delta(H)=B_{\rho+\delta}(0)$ and
the conclusion follows from Theorems~\ref{T1} and~\ref{T2}.
\end{proof}

The following theorem gives another reason for the interest in those  pseudosolutions of Eq.~\eqref{pode} which lie in a given ball around the origin.

\begin{theorem}
\label{thm3}
Let $\rho>0$. Suppose that Eq.~\eqref{lde} has an exponential dichotomy on $[0,\infty)$ and there exists $L>0$ such that
\begin{equation}
	\label{sublingrowth}
|f(t,x)|\leq L|x|\qquad\text{for all $t\geq0$ and $x\in B_\rho(0)$}.
\end{equation}
Then,~\eqref{smallcond} implies that there exists $\varepsilon>0$ such  that if~$y$ is a pseudosolution of Eq.~\eqref{pode} on $[0,\tau)$ for some $\tau\in(0,\infty]$ with $\sigma_y\leq\varepsilon$, then Eq.~\eqref{pode} has a pseudosolution $z$ on $[0,\tau)$ which lies in $B_\rho(0)$ and has the same error function as~$y$, i.e.
\begin{equation}
	\label{samerror}
e_z(t)=e_y(t)\qquad\text{for all $t\in[0,\tau)$}.
\end{equation}
\end{theorem}

\begin{proof}
In view of~\eqref{smallcond}, the equation
\begin{equation}
\label{epseq}
\frac{2N}{\lambda}(L\rho+\varepsilon)=\rho
\end{equation}
has a unique solution~$\varepsilon$ given by
\begin{equation*}
\varepsilon:=\biggl(\frac{\lambda}{2N}-L\biggr)\rho>0.
\end{equation*}
Suppose that $y$ is a pseudosolution of~\eqref{pode} on $[0,\tau)$ for some $\tau\in(0,\infty]$ with $\sigma_y\leq\varepsilon$.
 Let $\mathcal C:=C([0,\tau),\mathbb R^n)$ denote the topological vector space of all continuous functions  $z\colon [0,\tau) \to \mathbb R^n$ equipped with the topology of uniform convergence on compact subsets of $[0,\tau)$. Let
\begin{equation}
\label{Sdef}
S:=\biggl \{\,z\in\mathcal C:  \ \text{$ \sup_{t\in[0,\tau)} |z(t)|\leq\rho$}\biggr \}.
\end{equation}
Clearly, $S$ is a nonempty, closed and convex subset of~$\mathcal C$.  For $z\in S$ and $t\in[0,\tau)$, set
\begin{align*}
(\mathcal F z)(t)&:=\int_0^t T(t,s)P(s)\left[\,f(s,z(s))-h_y(s)\,\right]\,ds\\
&\qquad-\int_t^\tau T(t,s)(I-P(s))\left[\,f(s,z(s))-h_y(s)\,\right]\,ds,
\end{align*}
with $h_y$ as in~\eqref{beq}.   
In view of \eqref{ed1}, \eqref{ed2}, \eqref{beq}, \eqref{hsmall} and~\eqref{sublingrowth}, we have for $z\in S$ and $t\in[0,\tau)$, \begin{align*}
|(\mathcal F z)(t)|&\le \int_0^t |T(t,s)P(s)|(\,L|z(s)|+|h_y(s)|\,)\,ds\\
&\qquad+\int_t^\tau |T(t,s)(I-P(s))|(\,L|z(s)|+|h_y(s)|\,)\,ds\\
&\leq\int_0^t Ne^{-\lambda(t-s)}(L\rho+\sigma_y)\,ds
+\int_t^\infty Ne^{-\lambda(s-t)}(L\rho+\sigma_y)\,ds\\
&\leq\frac{2N}{\lambda}(L\rho+\varepsilon)=\rho,
\end{align*}
where the last equality follows from~\eqref{epseq}.
Thus, $\mathcal F z$ is well-defined and $\mathcal F(S) \subset S$. 
It follows in a standard manner that $\mathcal F\colon S\rightarrow S$ is continuous. In view of~\eqref{Sdef}, the functions from the image set $\mathcal F(S)\subset S$ are uniformly bounded on $[0,\tau)$.  Take now an arbitrary compact subinterval $I\subset [0, \tau)$ and $z\in S$. It follows by differentiation that 
\begin{equation}\label{aux}
(\mathcal F z)'(t)=A(t)(\mathcal F z)(t)+f(t,z(t))-h_y(t), \qquad t\in[0,\tau).
\end{equation}
Hence, 
\[
\sup_{t\in I} |(\mathcal F z)'(t)| \le  \rho\max_{t\in I} |A(t)|+L\rho+\varepsilon.
\] 
We conclude that the derivatives of the functions in $\mathcal F(S)$ are uniformly bounded on each compact subinterval of $[0, \tau)$,
which implies that the functions in $\mathcal F(S)$ are  equicontinuous on every compact subinterval of $[0,\tau)$. Therefore, the closure of $\mathcal F(S)$ is compact. 
By the application of the Schauder--Tychonoff fixed point theorem (see, e.g., \cite[Chap.~I, p.~9]{Cop}), we conclude that there exists $z\in S$ such that $z=\mathcal F z$. From~\eqref{aux}, it follows that
  \begin{equation*}
z'(t)=A(t)z(t)+f(t,z(t))-h_y(t), \qquad t\in[0,\tau).
\end{equation*}
Hence (see~\eqref{beq}),  $h_z=h_y$ identically on $[0,\tau)$, and hence~\eqref{samerror} holds. Finally, \eqref{Sdef} shows that $z$ lies in $B_\rho(0)$.
 \end{proof}

In the following result, we point out another simple consequence of Theorems~\ref{T1} and~\ref{T2}.
\begin{corollary}\label{cor: general perturbation}
\label{thm4}
Suppose that Eq.~\eqref{lde} has an exponential dichotomy on $[0,\infty)$.
Assume that $f\colon [0, \infty)\times \mathbb R^n \to \mathbb R^n$ is continuously differentiable and that there exist $L_1\ge 0$ and $L_2>0$ such that 
\begin{equation}
	\label{lingrowth}
|f_x(t,x)|\leq L_1+L_2|x|\qquad\text{for all $t\geq0$ and $x\in\mathbb R^n.$}
\end{equation}
If 
\begin{equation}
\label{modsmallcond}
L_1<\frac{\lambda}{2N}
\end{equation} 
and 
\begin{equation}
\label{rhocond}
0<\rho<\frac{1}{L_2}\biggl(\frac{\lambda}{2N}-L_1\biggr),
\end{equation}
then Eq.~\eqref{pode} has the conditional Lipschitz shadowing property in $B_\rho(0)$.
Moreover, if instead of the existence of an exponential dichotomy, we assume that Eq.~\eqref{lde} has an exponential contraction or exponential expansion on $[0,\infty)$, then 
the conditional Lipschitz shadowing property of Eq.~\eqref{pode} in $B_\rho(0)$ holds under the weaker conditions
\begin{equation}\label{new1}
L_1<\frac{\lambda}{N}
\end{equation}
and
\begin{equation}\label{new2}
0<\rho<\frac{1}{L_2}\biggl(\frac{\lambda}{N}-L_1\biggr).
\end{equation}
\end{corollary}
\begin{proof} We give a proof of the first statement of the corollary. The proof of the second statement is similar and hence it is omitted.

In view of~\eqref{rhocond}, we have that 
\begin{equation*}
L_1+L_2\rho<\frac{\lambda}{2N}.
\end{equation*}
Choose $\delta>0$ such that
\begin{equation*}
L_1+L_2(\rho+\delta)<\frac{\lambda}{2N}.
\end{equation*}
From this and~\eqref{lingrowth}, we have
	for all $t\geq0$ and $x\in B_{\rho+\delta}(0)$,
	\begin{equation*}
|f_x(t,x)|\leq L_1+L_2|x|\leq L_1+L_2(\rho+\delta)<\frac{\lambda}{2N}. 
\end{equation*}
This implies that condition~\eqref{balllipsch} of Corollary~\ref{thm2} holds with $L:=L_1+L_2(\rho+\delta)$. Since~$L$ satisfies~\eqref{smallcond}, the Lipschitz shadowing property of Eq.~\eqref{pode} in $B_\rho(0)$ follows from the first statement of Corollary~\ref{thm2}.
\end{proof}

In the following example, we show the importance and the sharpness of the assumptions of Corollary~\ref{cor: general perturbation}.
\begin{example}\label{EXM}
Consider the scalar autonomous equation
\begin{equation}
\label{speceq}
x'=-x-x^2-\frac{1}{4}
=-\biggl(x+\frac{1}{2}\biggr)^2, 
\end{equation}
which is a  special case of~\eqref{pode} when $n=1$, $A(t)=-1$ and $f(t,x)=-x^2-\frac{1}{4}$ for $t\in [0, \infty)$ and $x\in\mathbb R$. Its linear part $x'=-x$ admits an exponential contraction with $T(t,s)=e^{-(t-s)}$, $N=1$ and $\lambda=1$. Moreover, $f$ satisfies~\eqref{lingrowth} with $L_1=0$ and $L_2=2$. Therefore, conditions~\eqref{new1} and~\eqref{new2} of Corollary~\ref{cor: general perturbation} reduce to $0<\rho<1/2$. It follows from Corollary~\ref{thm4} that~\eqref{speceq} has the conditional Lipschitz shadowing in $B_\rho(0)=[-\rho, \rho]$ for any $\rho \in (0, 1/2)$.
We will show the importance of the condition $\rho<1/2$ by proving that the conditional Lipschitz shadowing property for Eq.~\eqref{speceq} in $B_{1/2}(0)=[-1/2, 1/2]$ does not hold. Suppose, for the sake of contradiction, that~\eqref{speceq} has the conditional Lipschitz shadowing property in $B_{1/2}(0)$ with some constants $\varepsilon_0, \kappa >0$. Choose
\begin{equation}
\label{deltasmall}
\delta\in\left(0,\min\bigl\{1, \varepsilon_0^{1/2}, \kappa^{-1}\bigr\}\right)
\end{equation}
so that
\begin{equation}
\label{smallcons}
\varepsilon:=\delta^2<\varepsilon_0 \qquad\text{and}\qquad \kappa \delta^2<\delta.
\end{equation}
Let $y$ denote the unique solution of the initial value problem
\begin{equation}
\label{deltaeq}
y'=-\biggl(y+\frac{1}{2}\biggr)^2+\delta^2,\qquad y(0)=-\frac{1}{2}.
\end{equation}
Observe that 
\begin{equation}
\label{yformula}
y(t)=-\frac{1}{2}+\frac{1-e^{-2\delta t}}{1+e^{-2\delta t}}\delta\qquad\text{for $t\ge 0$}.
\end{equation}
Hence,
\begin{equation}
\label{ylim}
\lim_{t\to\infty}y(t)=-\frac{1}{2}+\delta,
\end{equation}
while~\eqref{deltasmall}, \eqref{smallcons}, \eqref{deltaeq} and~\eqref{yformula} imply that
\begin{equation}
\label{ybounds}
-\frac{1}{2}\leq y(t)\le -\frac{1}{2}+\delta<\frac{1}{2}\qquad\text{for all $t\in[0,\infty)$,}
\end{equation}
and
\begin{equation}
\label{epsapprox}
\sigma_y=\sup_{t\geq0}\biggl|\,y'(t)+\biggl(y(t)+\frac{1}{2}\biggr)^2\,\biggr|=\delta^2=\varepsilon.
\end{equation}
This shows that $y\colon [0, \infty) \to \mathbb R$ is a pseudosolution of~\eqref{speceq} on $[0,\infty)$ such that $\sigma_y=\varepsilon<\varepsilon_0$ (see~\eqref{smallcons}) and $y(t)\in B_{1/2}(0)$ for $t\ge 0$. Hence, there exists a solution $x\colon [0, \infty)\to \mathbb R$ of~\eqref{speceq}  such that
\begin{equation}
\label{deltaclosesol}
	\sup_{t\geq0}|x(t)-y(t)|\leq \kappa \varepsilon=\kappa \delta^2.
\end{equation}
It follows by elementary calculations that if $c\in\mathbb R$, then the unique noncontinuable solution $x$ of~\eqref{speceq} with $x(0)=c$ is given by
\begin{equation*}
\label{solformula}
x(t)=-\frac{1}{2}+\frac{\gamma_c}{\gamma_c t+1}\qquad\text{with $\gamma_c:=c+\frac{1}{2}$}
\end{equation*}
for $t\in I_c$, where $I_c=(-\infty,\infty)$ for $c=-\frac{1}{2}$, $I_c=\bigl(-\frac{1}{\gamma_c},\infty\bigr)$ for $c>-\frac{1}{2}$ and
$I_c=\bigl(-\infty,-\frac{1}{\gamma_c}\bigr)$ for $c<-\frac{1}{2}$. Since  the solution $x$ satisfying~\eqref{deltaclosesol} is defined on $[0, \infty)$, the last possibility is excluded, and thus $\lim_{t\to\infty}x(t)=-\frac{1}{2}$. Hence (see~\eqref{ylim}),  $\lim_{t\to\infty}(y(t)-x(t))=\delta$. This, together with~\eqref{smallcons} and~\eqref{deltaclosesol} implies that 
\begin{equation*}
\delta=\lim_{t\to\infty}|x(t)-y(t)|\leq\sup_{t\geq0}|x(t)-y(t)|\leq \kappa \delta^2<\delta,
\end{equation*}
which is a contradiction.  Thus, \eqref{speceq} does not have the conditional Lipschitz shadowing property in $B_{1/2}(0)$. 
%This shows that the assumptions in the statement of Corollary~\ref{thm4} (see also Remark~\ref{imp}) are, in some sense, sharp.

\end{example}

\begin{remark}\label{remark: general perturbations}
It is easy to see that a conclusion similar to the one obtained in Corollary \ref{cor: general perturbation} holds true for more general perturbations of Eq.~\eqref{lde}. For instance, if Eq.~\eqref{lde} has an exponential dichotomy on $[0,\infty)$, $L_1$ satisfies \eqref{modsmallcond} and $f$ is such that 
\begin{equation*}
	|f_x(t,x)|\leq L_1+L_2|x|+L_3|x|^2+\ldots+L_{k+1}|x|^k\qquad\text{for all $t\geq0$ and $x\in\mathbb R^n$,}
\end{equation*}
with $k\in \mathbb{N}$ and $L_j\geq 0$ for every $j\in \{1,\ldots, k+1\}$, then for $\rho>0$ small enough such that \[L_1+L_2\rho+L_3\rho^2+\ldots+L_{k+1}\rho^k< \frac{\lambda}{2N},\]
Eq.~\eqref{pode} has the conditional Lipschitz shadowing property in $B_\rho(0)$.
\end{remark}

\section{Conditional Lipschitz shadowing via the logarithmic norm}\label{lnorms}
In this section, we give sufficient conditions for the conditional Lipschitz shadowing of Eq.~\eqref{ODE}. These conditions will be formulated in terms of the logarithmic norm~$\mu$ defined by~\eqref{lognorm}. Let us recall some useful properties of~$\mu$ from
~\cite[p.~41]{Cop}. For every $\alpha\geq0$ and $A, B\in \mathbb R^{n\times n}$, we have 
\begin{align}
\mu(\alpha A)&=\alpha\mu(A),\label{muposhom}\\
|\mu(A)|&\leq|A|,\label{munormineq}\\	
\mu(A+B)&\leq\mu(A)+\mu(B),\label{musub}\\
|\mu(A)-\mu(B)|&\leq|A-B|.\label{mucont}
\end{align}
The values of $|A|$ and $\mu(A)$ for the most commonly used norms
\begin{equation*}
|x|_\infty=\max_i|x_i|,\qquad	
|x|_1=\sum_i|x_i|,\qquad
|x|_2=\biggl(\sum_i|x_i|^2\biggr)^{1/2},
\end{equation*}
in $\mathbb R^n$ are given by
\begin{equation*}
|A|_\infty=\max_i\sum_k|a_{ik}|,\qquad	
|A|_1=\max_k\sum_i|a_{ik}|,\qquad	
|A|_2=\sqrt{s(A^TA)},
\end{equation*}
and
	\begin{equation*}
\mu_\infty(A)=\max_i\biggl(a_{ii}+\sum_{k,\,k\neq i}|a_{ik}|\biggr),\quad
\mu_1(A)=\max_k\biggl(a_{kk}+\sum_{i,\,i\neq k}|a_{ik}|\biggr),\quad
\mu_2(A)=s\biggl(\frac{A^T+A}{2}\biggr),
\end{equation*}
where $s(A^TA)$ and $s((A^T+A)/2)$ are the largest (real) eigenvalue of $A^TA$ and $(A^T+A)/2$, respectively.
We will also need the following auxiliary result.
\begin{lemma}
For any continuous map $M\colon[0,1]\rightarrow\mathbb R^{n\times n}$, we have that 
\begin{equation}
\label{muintineq}
\mu\biggl(\int_0^1 M(s)\,ds\biggr)\leq\int_0^1\mu\left(M(s)\right)\,ds.
\end{equation}
\end{lemma}

\begin{proof}
%[Proof of the lemma]
The existence of the integral on the right-hand side of~\eqref{muintineq} is a consequence of the continuity of $\mu:\mathbb R^{n\times n}\rightarrow\mathbb R$ (see~\eqref{mucont}).
In order to prove~\eqref{muintineq}, observe that the subadditivity and positive homogeneity of~$\mu$ (see~\eqref{muposhom} and~\eqref{musub}), applied to the integral sums $\sum_{i=1}^k M(s_i)\Delta s_i$,
where $0=s_0<s_1<\dots<s_k=1$ is a partition of $[0,1]$ and $\Delta s_i:=s_i-s_{i-1}$ for $i=1,\dots,k$, imply that
\begin{equation*}
\mu\biggl(\sum_{i=1}^k M(s_i)\Delta s_i\biggr)\leq\sum_{i=1}^k\mu\left(M(s_i)\right)\Delta s_i.
\end{equation*}
Letting $\Delta:=\max_{1\leq i\leq k}\Delta s_i\rightarrow0$ and using the continuity of~$\mu$ again, we conclude that~\eqref{muintineq} holds.
\end{proof}

Now we can state and prove the main result of this section which provides a sufficient condition under which Eq.~\eqref{ODE} has the conditional Lipschitz shadowing property in a given set $H\subset\mathbb R^n$.
\begin{theorem}
\label{globcrit}
Let $\emptyset\neq H\subset\mathbb R^n$ and suppose that~$g\colon [0, \infty) \times \mathbb R^n\rightarrow\mathbb R^n$ is a continuously differentiable function. If there exist $\delta>0$ and $m>0$ such that
	\begin{equation}
	\label{munegbound}
\mu(g_x(t, x))\leq-m\qquad\text{for all $t\ge 0$ and $x\in\mathcal N_\delta(H)$},
\end{equation}
then Eq.~\eqref{ODE} has the conditional Lipschitz shadowing property in~$H$.
\end{theorem}

\begin{proof}
We will show that Eq.~\eqref{ODE} has the conditional Lipschitz shadowing property in~$H$ with $\varepsilon_0=m\delta$ and $\kappa=m^{-1}$. Suppose that $0<\varepsilon\leq\varepsilon_0=m\delta$ and $y$ is a pseudosolution of Eq.~\eqref{ODE} on $[0,\tau)$ for some $\tau\in(0,\infty]$ such that $\sigma_y=\sup_{0\leq t<\tau}|y'(t)-g(t, y(t))|\leq\varepsilon$ and $y(t)\in H$ for all $t\in[0,\tau)$. Let $x$ be the noncontinuable solution of Eq.~\eqref{ODE} with initial value $x(0)=y(0)$. It is known (see, e.g., \cite[Chap.~I (III), p.~16]{Cop})	
 that $x$ is defined on $[0,\sigma)$ for some $\sigma\in(0,\infty]$
and $\sigma=\infty$ whenever $x$ is bounded. Let $\omega:=\min\{\sigma,\tau\}$. Define
\begin{equation}
\label{zdef}
z(t):=x(t)-y(t)\qquad\text{for $t\in[0,\omega)$}.
\end{equation}
We claim that
\begin{equation}
\label{zbound}
|z(t)|<\frac{\varepsilon}{m}\qquad\text{for all $t\in[0,\omega)$}.
\end{equation}
Suppose, for the sake of contradiction, that~\eqref{zbound} does not hold. Since $z(0)=x(0)-y(0)=0$, there exists $t_1\in(0,\omega)$ such that
\begin{equation}
\label{tonedef}
|z(t)|<\frac{\varepsilon}{m}\quad\text{for all $t\in[0,t_1)$}\qquad\text{and}\qquad |z(t_1)|=\frac{\varepsilon}{m}.
\end{equation}
From~\eqref{ODE} and~\eqref{zdef}, we find for $t\in[0,\omega)$,
\begin{equation*}
z'(t)=g(t, y(t)+z(t))-y'(t)=g(t, y(t)+z(t))-g(t, y(t))+k_y(t),
\end{equation*}
where
\begin{equation*}
k_y(t):=g(t, y(t))-y'(t)\qquad\text{for $t\in[0,\omega)$}. 
\end{equation*}
Hence
\begin{equation}
\label{zdiffeq}
z'(t)=A(t)z(t)+k_y(t)\qquad\text{for $t\in[0,\omega)$},
\end{equation}
where
\begin{equation}
\label{Adef}
A(t):=\int_0^1 g_x(t, y(t)+sz(t))\,ds\qquad\text{for $t\in[0,\omega)$}.
\end{equation}
From~\eqref{muintineq} and~\eqref{Adef}, we obtain
\begin{equation}
\label{muAineq}
\mu\left(A(t)\right)\leq\int_0^1 \mu(g_x(t, y(t)+sz(t)))\,ds \qquad\text{for $t\in[0,\omega)$}.
\end{equation}
Let $t\in[0,t_1]$ be fixed. In view of~\eqref{tonedef}, for every $s\in[0,1]$, we have
\begin{equation*}
|\left(y(t)+sz(t)\right)-y(t)|=s|z(t)|\leq|z(t)|\leq\frac{\varepsilon}{m}\leq\frac{\varepsilon_0}{m}=\delta.
\end{equation*}
Hence, for every $t\in[0,t_1]$ and $s\in[0,1]$, we have that $y(t)\in H$ and 
$y(t)+sz(t)\in B_\delta(y(t))\subset\mathcal N_\delta(H)$.
This, together with~\eqref{munegbound} and~\eqref{muAineq}, yields
\begin{equation}
\label{mukeyneq}
\mu\left(A(t)\right)\leq-m\qquad\text{for all $t\in[0,t_1]$}.
\end{equation}
Let $T(t,s)$ denote the transition matrix 
of the homogeneous linear differential equation~\eqref{lde}, where $A(t)$ is given by~\eqref{Adef}. Then, for every $s\in[0,\omega)$ and $\xi\in\mathbb R^n$, the solution of Eq.~\eqref{lde} with initial value~$\xi$ at $t=s$ is given by $x(t)=T(t,s)\xi$ for $t\in[0,\omega)$.
By Coppel's inequality \cite[Chap.~III, Theorem~3, p.~58]{Cop}, we have for $0\leq s\leq t<\omega$,
\begin{equation*}
|T(t,s)\xi|\leq\exp\biggl(\int_s^t\mu(A(u))\,du\biggr)|\xi|.
\end{equation*}
Since $\xi\in\mathbb R^n$ was arbitrary, this implies that
\begin{equation}
\label{transmatgrowth}
|T(t,s)|=\sup_{0\neq\xi\in\mathbb R^n}\frac{|T(t,s)\xi|}{|\xi|}\leq\exp\biggl(\int_s^t\mu(A(u))\biggr)\,du
\quad\text{whenever $0\leq s\leq t<\omega$}.
\end{equation}
Since $z$ is a solution of the nonhomogeneous equation~\eqref{zdiffeq} with initial value $z(0)=0$, by the variation of constants formula, we have
\begin{equation*}
z(t)=\int_0^t T(t,s)k_y(s)\,ds\qquad\text{for all $t\in[0,\omega)$}.
\end{equation*}
From this, \eqref{mukeyneq} and~\eqref{transmatgrowth}, and taking into account that $\sup_{0\leq t<\omega}|k_y(t)|\leq\sigma_y\leq\varepsilon$, we obtain
\begin{align*}
|z(t_1)|&\leq\int_0^{t_1} |T(t_1,s)||k_y(s)|\,ds\leq\varepsilon\int_0^{t_1} \exp\biggl(\int_s^{t_1}\mu(A(u))\biggr)\,du\\
&\leq\varepsilon\int_0^{t_1}e^{-m(t_1-s)}\,ds=\frac{\varepsilon}{m}(1-e^{-mt_1})<\frac{\varepsilon}{m}.
\end{align*}
This contradicts~\eqref{tonedef} and hence~\eqref{zbound} holds.

 Next we show that $\sigma\geq\tau$. Otherwise, $0<\sigma<\tau$ and hence $\omega=\sigma$. This, together with~\eqref{zdef} and \eqref{zbound}, implies that for all $t\in[0,\sigma)$, 
\begin{equation*}
|x(t)|=|y(t)+z(t)|\leq|y(t)|+|z(t)|\leq\max_{0\leq t\leq\sigma}|y(t)|+\frac{\varepsilon}{m}.
\end{equation*}
Consequently, $x$ is bounded on $[0,\sigma)$ and hence $\sigma=\infty$ contradicting the fact that $\sigma<\tau$. Thus, $\sigma\geq\tau$ and hence $\omega=\tau$. This, together with~\eqref{zdef} and \eqref{zbound}, implies that condition~\eqref{closesol} is satisfied with $\kappa=m^{-1}$. The proof of the theorem is completed.
\end{proof}

\begin{example} (Example~\ref{EXM} revisited) \label{example 1 revisited}
We note that Eq.~\eqref{speceq} is a special case of~\eqref{ODE} with
\begin{equation*}
g(t,x)=-\bigg (x+\frac  1 2\bigg )^2, \qquad  t\ge 0, \ x\in \mathbb R.
\end{equation*}
In Example~\ref{EXM} we have shown that Eq.~\eqref{speceq} has the conditional Lipschitz shadowing property in $[-\rho, \rho]$,  for every $0<\rho <\frac 1 2$.
From Theorem~\ref{globcrit}, we can deduce a stronger result showing that the interval $[-\rho,\rho]$ with 
$\rho\in\left(0,\frac{1}{2}\right)$ can be replaced with the larger interval $[-\rho,\infty)$. Indeed, 
as already noted, in the scalar case, we have that $\mu(A)=A$, and hence condition~\eqref{munegbound} reduces to
	\begin{equation}
	\label{scalmunegbound}
g_x(t,x)\leq-m<0\qquad\text{for all $x\in\mathcal N_\delta(H)$}.
\end{equation}
Let $H:=[-\rho,\infty)$, where $0<\rho<\frac{1}{2}$. Choose $\delta\in(0,\frac{1}{2}-\rho)$. Then, for all $x\in\mathcal N_\delta(H)=[-\rho-\delta,\infty)$, 
	\begin{equation*}
g_x(t,x)=-2\biggl(x+\frac{1}{2}\biggr)\leq-2\biggl(-\rho-\delta+\frac{1}{2}\biggr)<0, 
\end{equation*}
which shows that condition~\eqref{scalmunegbound} is satisfied with $m:=2(-\rho-\delta+\frac{1}{2})>0$. By the application of Theorem~\ref{globcrit}, we conclude that, for every $\rho\in (0,\frac{1}{2})$, Eq.~\eqref{speceq} has the 
conditional Lipschitz shadowing property in $[-\rho,+\infty)$. Since the result obtained in Example~\ref{EXM} implies that the conditional Lipschitz shadowing property for~\eqref{speceq} in $[-\frac{1}{2},+\infty)$ does not hold, this is the best result which can be achieved.
\end{example}

The following corollary of Theorem~\ref{globcrit} for $H=\mathbb R^n$ provides a new criterion for the standard Lipschitz shadowing property of Eq.~\eqref{ODE} and hence it is interesting itself.

\begin{corollary}\label{abb}
Suppose that~$g\colon [0, \infty) \times \mathbb R^n\rightarrow\mathbb R^n$ is a continuously differentiable function such that for some $m>0$,
\begin{equation*}
\mu(g_x(t, x))\leq-m\qquad\text{for all $t\ge 0$ and $x\in\mathbb R^n$}.
\end{equation*}
Then, Eq.~\eqref{ODE} has the  Lipschitz shadowing property.
\end{corollary}

Now we present a simple corollary of Theorem~\ref{globcrit} for the autonomous equation 
\begin{equation}
\label{aut}
	x'=h(x),
\end{equation}
where $h\colon\mathbb R^n\rightarrow\mathbb R^n$ is continuously differentiable. 
\begin{corollary}\label{modglobcrit}
Let $H$ be a nonempty, bounded subset of~$\mathbb R^n$. If $h\colon  \mathbb R^n\rightarrow\mathbb R^n$ is a continuously differentiable function such that
\begin{equation}
\label{musupbound}
\sup_{x\in H}\mu(h'( x))<0,
\end{equation}
then Eq.~\eqref{aut} has the conditional Lipschitz shadowing property in~$H$.
\end{corollary}

\begin{proof}
Eq.~\eqref{aut} is a special case of~\eqref{ODE} with $g(t,x):=h(x)$ for $t\geq0$ and $x\in\mathbb R^n$.
Choose $\varepsilon\in(0,-k)$, where $k:=\sup_{x\in H}\mu(h'(x))<0$ (see~\eqref{musupbound}). Since~$H$ is bounded, there exists $\rho>0$ such that $H\subset B_\rho(0)$. The continuity of~$\mu$ and~$h'$ implies that $\mu\circ h'$ is uniformly continuous on the compact set $B_{\rho+1}(0)$. Therefore, there exists $\delta\in(0,1)$ such that
\begin{equation}
\label{uncont}
|\mu(h'(x))-\mu(h'(\tilde x))|\leq\varepsilon\qquad\text{whenever $x,\tilde x\in B_{\rho+1}(0)$ and $|x-\tilde x|\leq\delta$}.
\end{equation}
Let $x\in\mathcal N_\delta(H)$. Then, there exists $\tilde x\in H$ such that $|x-\tilde x|\leq\delta$.
Hence, $\tilde x\in H\subset B_\rho(0)$ and $|x|\leq|\tilde x|+|x-\tilde x|\leq\rho+\delta<\rho+1$. Thus, we have that $x,\tilde x\in B_{\rho+1}(0)$ and $|x-\tilde x|\leq\delta$. From~\eqref{uncont} and the definition of~$k$, we obtain that 
\begin{equation*}
	\mu(g_x(t,x))=\mu(h'(x))\leq\mu(h'(\tilde x))+\varepsilon\leq k+\varepsilon.
\end{equation*}
Since $x\in\mathcal N_\delta(H)$ was arbitrary, condition~\eqref{munegbound} is satisfied with $m=-(k+\varepsilon)>0$. The desired conclusion now follows readily from Theorem~\ref{globcrit}.
\end{proof}

Finally, we illustrate 
the importance of assumption~\eqref{musupbound} of Corollary~\ref{modglobcrit} in a special case of a classic model from epidemiology.
\begin{example}\label{epi}
Consider the system
\begin{equation}
\label{SIspec} 
 \begin{split}
  S' &= 1-IS-S, \\[0mm]
  I' &=  IS-I,
 \end{split}
\end{equation}
which is a special case of the modified \emph{Kermack--McKendrick equation} (see~\cite[Chap.~2, Sec.~2.3, p.~53]{Li}). Biologically meaningful solutions are generated by initial data $(S(0),I(0))$ from the set
\begin{equation*}
\label{gammaset}
\Gamma:= \left\{(S,I) \in \mathbb{R}^2:S\geq0,\,I\geq0,\,S+I \le 1 \right\}.
\end{equation*}
For $c\in [0,1)$, define
\begin{equation*}
\label{gammacset}
\Gamma_c:= \left\{(S,I) \in \mathbb{R}^2:S\geq0,\,I\geq0,\,S+I \le 1-c\right\}.
\end{equation*}
Observe  that $\Gamma_0=\Gamma$ and $\Gamma_c\subset\Gamma_0$ for $c\in[0,1)$.
Eq.~\eqref{SIspec} is a special case of Eq.~\eqref{aut} where $h\colon\mathbb R^2\rightarrow\mathbb R^2$ is given by
\begin{equation*}
\label{choice}
h(S,I)=(1-IS-S,IS-I)^T\qquad\text{for $(S,I)^T\in\mathbb R^2$}. 
\end{equation*}
We will show that, for every $c\in(0,1)$, Eq.~\eqref{SIspec} has the conditional Lipschitz shadowing property in~$\Gamma_c$, but the same property in~$\Gamma_0$ does not hold. 
%Let $c\in[0,1)$ be arbitrary. 
Evidently, $h$ is continuously differentiable and
\begin{equation*}
h'(S,I)=
 \begin{pmatrix}
-I-1 & -S \\
I & S-1 \\
\end{pmatrix}
\qquad\text{for $(S,I)^T\in\mathbb R^2$}.
\end{equation*}
Hence,
\begin{equation*}
\mu_\infty(h'(S,I))=\max\{\,-I-1+|S|,S-1+|I|\,\}
\qquad\text{for $(S,I)^T\in\mathbb R^2$}.
\end{equation*}
From this and the definition of~$\Gamma_c$, we obtain
\begin{equation}
\label{mugderineq}
\mu_\infty(h'(S,I))=S-1+I\leq-c\qquad\text{for all $(S,I)^T\in\Gamma_c$},
\end{equation}
where $c\in[0,1)$ is arbitrary.
This shows that if $c\in(0,1)$, then condition~\eqref{musupbound} of Corollary~\ref{modglobcrit} is satisfied with $\mu=\mu_\infty$ and $H=\Gamma_c$. By the application of 
Corollary~\ref{modglobcrit}, we conclude that, for every $c\in(0,1)$, Eq.~\eqref{SIspec} has the conditional Lipschitz shadowing property in~$\Gamma_c$. 

Next we show that the same property in $\Gamma_0$ does not hold. Suppose, for the sake of contradiction, that Eq.~\eqref{SIspec} has the conditional Lipschitz shadowing property in~$\Gamma_0$. Note that the definition of the conditional Lipschitz shadowing is independent of the norm used since all norms on $\mathbb R^n$ are equivalent. Therefore, we may (and do) use the infinity norm $|\cdot|_\infty$ on~$\mathbb R^2$. Let $\varepsilon_0, \kappa >0$ be the constants from the definition of the conditional Lipschitz shadowing property in~$\Gamma_0$. It is easily verified that for every $\varepsilon>0$, 
\begin{equation*}
	P_\varepsilon:=(1-\sqrt{\varepsilon},\sqrt{\varepsilon})^T\in\mathbb R^2
\end{equation*}
is an equilibrium of the system
\begin{equation*}
 \begin{split}
  S' &= 1-IS-S-\varepsilon, \\[0mm]
  I' &= IS - I+\varepsilon,
 \end{split} 
\end{equation*}
which is a perturbation of Eq.~\eqref{SIspec}. 
%Consequently, 
Hence, for 
every 
$\varepsilon\in(0,\min\{\varepsilon_0,1\})$, $P_\varepsilon=(1-\sqrt{\varepsilon},\sqrt{\varepsilon})^T$ is a constant pseudosolution of~\eqref{SIspec} on $[0,\infty)$ with maximum error $\sigma_{P_\varepsilon}=\varepsilon\leq\varepsilon_0$ and such that $P_\varepsilon\in\Gamma_0$. By the definition of the conditional Lipschitz shadowing, this implies that, for every $\varepsilon\in(0,\min\{\varepsilon_0,1\})$, Eq.~\eqref{SIspec} has a solution $(S_\varepsilon(t),I_\varepsilon(t))^T$ on $[0,\infty)$ such that 
\begin{equation}
\label{siepsloc}
(S_{\varepsilon}(t),I_\varepsilon(t))^T\in B_{\kappa\varepsilon}(P_\varepsilon)=B_{\kappa\varepsilon}((1-\sqrt{\varepsilon},\sqrt{\varepsilon})^T)\qquad\text{for all $t\in[0,\infty)$}.
\end{equation}
In particular, we have that
\begin{equation}
\label{limsetincl}
\emptyset\neq\omega(S_\varepsilon,I_\varepsilon)\subset B_{\kappa\varepsilon}(P_{\varepsilon})\qquad\text{whenever $\varepsilon>0$ is sufficiently small},
\end{equation}
where $\omega(S_\varepsilon,I_\varepsilon)$ denotes the omega-limit set of the solution $(S_\varepsilon(t),I_\varepsilon(t))^T$. If $\varepsilon>0$ is sufficiently small, then  $\kappa\varepsilon<\sqrt{\varepsilon}$. Hence,
\begin{equation}
\label{bepsloc}
B_{\kappa\varepsilon}(P_{\varepsilon})\subset
G:=(0,1)\times(0,1)\qquad\text{whenever $\varepsilon>0$ is sufficiently small.}
\end{equation}
%(see Firgure~\ref{fig:nobfigure}). 
%\begin{figure}[h]
%  \centering
%  \includegraphics[width=0.35\textwidth]{fig1.eps}
 % \caption{Graphs of the functions $b=\zeta(a)$ (red) and $b=\frac{\pi}{\nu(a)}$ (blue).}
 % \label{fig:nobfigure}
%\end{figure
Choose $\varepsilon>0$ small enough such that both~\eqref{siepsloc} and~\eqref{bepsloc} are satisfied. Define $V\colon\mathbb R^2\rightarrow\mathbb R$ by
\begin{equation*}
V(S,I)=I\qquad\text{for $(S,I)^T\in\mathbb R^2$}.
\end{equation*}
Then $V'(S,I)=(0,1)$
and for the derivative of~$V$ along system~\eqref{SIspec}, we have
\begin{equation*}
\dot V_{\eqref{SIspec}}(S,I)=V'(S,I)h(S,I)=-I(1-S)\leq0\qquad\text{for $(S,I)^T\in\overline{G}=[0,1]\times[0,1]$}.
\end{equation*}
Thus, $V$ is a Lyapunov function for Eq.~\eqref{SIspec} on $\overline{G}$ 
(see \cite[Chap.~2, Definition~6.1, p.~30]{LaSa}) and
\begin{equation}
\label{esetdef}
E:=\{\,(S,I)^T\in\overline{G}\mid\dot V_{\eqref{SIspec}}(S,I)=0\,\}=
\{\,(S,I)^T\in\overline{G}\mid\text{$I=0$ or $S=1$}\,\}.
\end{equation}
By the application of LaSalle's invariance principle (see, e.g., \cite[Chap.~2, Theorem~6.1, p.~30]{LaSa}), we conclude that 
$\omega(S_\varepsilon,I_\varepsilon)\subset E$. This, combined with~\eqref{limsetincl}, yields
\begin{equation*}
	\emptyset\neq\omega(S_\varepsilon,I_\varepsilon)\subset B_{\kappa\varepsilon}(P_{\varepsilon})\cap E.
\end{equation*}
Thus, $B_{\kappa\varepsilon}(P_{\varepsilon})\cap E\neq\emptyset$. On the other hand,~\eqref{bepsloc} and~\eqref{esetdef} imply that 
$B_{\kappa\varepsilon}(P_{\varepsilon})\cap E=\emptyset$.	This contradiction proves that Eq.~\eqref{SIspec} does not have the conditional shadowing property in~$\Gamma_0$. Note that if we take $H=\Gamma_0$ and $\mu=\mu_\infty$, then~\eqref{mugderineq} with $c=0$ implies that
\begin{equation*}
\sup_{(S,I)^T\in\Gamma_0}\mu_\infty(h'(S,I))=0,
\end{equation*}
which shows the importance and the sharpness of condition~\eqref{musupbound} in Corollary~\ref{modglobcrit}.
\end{example}

\section*{Acknowledgements}
L. Backes was
 partially 
supported by the CNPq-Brazil PQ fellowship under Grant No. 307633/2021-7.
D. Dragi\v cevi\' c was supported in part by Croatian Science Foundation under the Project IP-2019-04-1239 and by the University of Rijeka under the Projects uniri-prirod-18-9 and uniri-prprirod-19-16. 
M. Onitsuka was supported in part by the Japan Society for the Promotion of Science (JSPS) KAKENHI Grant No. JP20K03668.
M.~Pituk was supported by the Hungarian National Research, Development and Innovation Office Grant No.~K139346.

\end{document}